\documentclass[a4paper,11pt]{article}

\input{styleart.sty}
\date{}

\begin{document}
\author{Rizos Sklinos}
\title{A note on ampleness in the theory of non abelian free groups}

\maketitle

\begin{abstract}
Recently Ould Houcine-Tent \cite{AmpleOK} proved that the theory of non abelian free groups is $n$-ample for any $n<\omega$. 
We give a sequence of primitive elements in $\F_{\omega}$ witnessing the above mentioned result. 
Our proof is not independent from \cite{AmpleOK} as we essentially use some theorems from there. On the other hand our 
witnessing sequence is much simpler.
\end{abstract}

\section{Introduction}
The notion of $n$-ampleness, for some natural number $n$, fits in the general context of geometric stability theory. As the definition 
may look artificial or technical, we first give the historical background of its 
development. We start by working in a vector space $V$ and we consider two finite dimensional subspaces $V_1,V_2\leq V$. Then one can see 
that $dim(V_1 + V_2) = dim(V_1) + dim(V_2) - dim(V_1\cap V_2)$, and the point really is that $V_1$ is linearly independent from 
$V_2$ over $V_1\cap V_2$. In an abstract stable theory the notion of linear independence is replaced by forking independence and the 
above property gives rise to the notion of $1$-basedness. A stable theory $T$ is {\em $1$-based} if there are no $a$, 
$b$ such that $acl^{eq}(a)\cap acl^{eq}(b)=acl^{eq}(\emptyset)$ and $a$ forks with $b$ over $\emptyset$. 
The notion of $1$-basedness turned out to be very fruitful in model theory and one of the major 
results concerning this notion was the following theorem by Hrushovski-Pillay \cite{WeakNormal}.

\begin{thm}\label{HruPill}
Let $\mathcal{G}$ be a $1$-based stable group. Then every definable set $X\subseteq G^n$ is a Boolean combination
of cosets of almost $\emptyset$-definable subgroups of $G^n$. Moreover $G$ is abelian-by-finite.
\end{thm}

On the other hand, Hrushovski's seminal work in refuting Zilber's trichotomy conjecture (see \cite{HrushovskiSMSet})  
produced ``new'' strongly minimal sets that had an interesting property. Hrushovski isolated this property and 
called it $CM$-triviality (for Cohen-Macaulay). 
A stable theory $T$ is {\em $CM$-trivial} if there are no $a, b, c$ such that 
$a$ forks with $c$ over $\emptyset$, $a$ is independent from $c$ over $b$, $acl^{eq}(a)\cap acl^{eq}(b)=acl^{eq}(\emptyset)$ 
and finally $acl^{eq}(a,b)\cap acl^{eq}(a,c)=acl^{eq}(a)$. A kind of an analogue to the moreover statement of 
the above theorem has been proved by Pillay in \cite{PillFMR}.  

\begin{thm} 
A $CM$-trivial group of finite Morley rank is nilpotent-by-finite.
\end{thm}

Pillay first realized the pattern and proposed an hierarchy of ampleness, non $1$-basedness ($1$-ampleness) and non $CM$-triviality ($2$-ampleness) 
being the first two items in it (see \cite{PiAmp}). His definition needed some fine ``tuning'' as observed by Evans \cite{EvAmp}. 

\begin{defi}[\cite{EvAmp}]\label{Ample}
Let $T$ be a stable theory and $n\geq1$. Then $T$ is $n$-ample if (after possibly adding some parameters) there are 
$a_0,a_1,\ldots,a_n$ such that:
\begin{enumerate}
\item $a_0$ forks with $a_n$ over $\emptyset$;
\item $a_{i+1}$ does not fork with $a_0,\ldots,a_{i-1}$ over $a_i$, for $1\leq i<n$;
\item $acl^{eq}(a_0)\cap acl^{eq}(a_1)=acl^{eq}(\emptyset)$;
\item $acl^{eq}(a_0,\ldots,a_{i-1},a_i)\cap acl^{eq}(a_0,\ldots,a_{i-1},a_{i+1})=acl^{eq}(a_0,\ldots,a_{i-1})$, for $1\leq i<n$.
\end{enumerate}

\end{defi}

The purpose of this paper is to give an alternative sequence to the one given in \cite{AmpleOK} witnessing $n$-ampleness, for any $n<\omega$, 
in the theory of non abelian free groups. The advantage of our sequence is that it is much simpler and consists only of 
primitive elements (instead of triples of elements). 
Though the witnessing sequence is the only point that we diverge from \cite{AmpleOK} we try to make the paper self-contained.

The paper is organized as follows.
In section \ref{2}, we give some background around the main geometric tool, i.e. $JSJ$-decompositions. 

In section \ref{3} we analyze the three main ingredients of the proof, i.e. the geometric elimination of imaginaries \cite{SelaImaginaries}, 
the understanding of algebraic closure \cite{AlgOV,AmpleOK}, and finally the understanding of forking independence 
for primitive elements \cite{PillayForking,PillayGenericity}, in non abelian free groups.

In the final section we give explicitly a sequence witnessing $n$-ampleness for any $n<\omega$.

Our notation is standard. By $\F_n$ we denote the free group of rank $n$, and by $T_{fg}$ the common theory of non abelian free groups. 
If $\mathbb{M}$ is a ``big'' saturated model of 
a first order theory $T$ and $A\subset \mathbb{M}$, then by $acl_{\mathbb{M}}(A)=acl(A)$ we mean the 
``real'' algebraic closure, while by $acl^{eq}_{\mathbb{M}}(A)=acl^{eq}(A)$ we denote the algebraic closure computed in $T^{eq}$.

\section{JSJ decompositions}\label{2}
In this section we will describe the notion of $JSJ$-decompositions. Roughly speaking, a $JSJ$-decomposition of 
a group $G$ is a splitting as a graph of groups in which one can ``read'' all possible splittings 
of $G$ over a given class, $\mathcal{A}$, of subgroups, i.e. splittings of $G$ where all edge groups belong 
to the class $\mathcal{A}$. Note that it is not immediate that such a splitting exists (provided $\mathcal{A}$ is given). 
Actually the existence of (cyclic) $JSJ$-decompositions 
for hyperbolic groups \cite{SelaHypII} by Sela and later for finitely presented groups \cite{RipsSelaJSJ} by Rips-Sela 
was a major breakthrough in group theory. After that, various results appeared mainly 
extending the class of subgroups $\mathcal{A}$, the more general being Fujiwara-Papazoglu \cite{FujiwaraPapasoglu}, 
extending Rips-Sela by taking $\mathcal{A}$ to be the class of slender subgroups, i.e. groups for which all subgroups are finitely generated.

We give a more formal account of the $JSJ$ theory following the unifying framework of Guirardel-Levitt (see \cite{GuirardelLevittJSJI},
\cite{GuirardelLevittJSJII}). We note that we will change our point of view from group splittings to 
groups acting on trees and the other way around (using the duality explained by Bass-Serre theory \cite{SerreTrees}) 
when it is convenient, but it will always be clear what we mean from the context.

Let $G$ be a group acting on a tree $T$ (by automorphisms and without inversions), we call $T$ a {\em cyclic $G$-tree} 
if all edge stabilizers are cyclic. If $H$ is a subgroup of $G$, then $H$ is {\em elliptic} in $T$ if it 
fixes a point in $T$, otherwise it is called {\em hyperbolic}. 
We fix a group $G$ and we work in the class, $\mathcal{T}_G$, of all cyclic $G$-trees. 
A tree in $\mathcal{T}_G$ is {\em universally elliptic} if its edge stabilizers are elliptic 
in every tree in $\mathcal{T}_G$. If $T_1,T_2$ are two trees in $\mathcal{T}_G$, 
we say that $T_1$ {\em dominates} $T_2$ if every subgroup of $G$ which is elliptic in $T_1$, is elliptic in $T_2$. 
 
A {\em cyclic $JSJ$-tree} is a universally elliptic tree which dominates 
any universally elliptic tree. We will be also interested in relative cyclic $JSJ$-trees of a group $G$ 
with respect to a family of subgroups $\mathcal{H}$. In this case we work in the class of all cyclic $G$-trees in 
which every $H\in\mathcal{H}$ is elliptic. Finally, by a {\em (relative) cyclic $JSJ$-decomposition} of a group $G$ we mean 
the corresponding graph of groups obtained by the action of $G$ on a (relative) cyclic $JSJ$-tree.
  
Let $G$ be a group and $\mathcal{H}$ be a family of subgroups of $G$, let $\mathcal{T}_{(G,\mathcal{H})}$ be the class of all cyclic $G$-trees 
in which every $H\in\mathcal{H}$ is elliptic. Let $T_{JSJ}$ be a relative cyclic $JSJ$-tree in $\mathcal{T}_{(G,\mathcal{H})}$. 
A vertex in $T_{JSJ}$ is called {\em rigid} if the vertex stabilizer is elliptic in any other tree in $\mathcal{T}_{(G,\mathcal{H})}$, 
and {\em flexible} if not. 

In this context the essential part of the $JSJ$-theory is the description of the flexible vertex stabilizers of a $JSJ$-tree 
(provided such a tree exists). We give the description in the case of torsion-free hyperbolic groups. 

We first recall that the fundamental group 
of a compact surface, $\Sigma$, with boundary is a free group. 
Each boundary component of $\Sigma$ has cyclic fundamental group, and gives rise in $\pi_1(\Sigma)$ to a 
conjugacy class of cyclic subgroups: we call these {\em maximal boundary subgroups}. 
A {\em boundary subgroup} of $\pi_1(\Sigma)$ is subgroup of a maximal boundary subgroup of $\pi_1(\Sigma)$. 

Now, let $G_u$ be a vertex stabilizer in a tree which is in $\mathcal{T}_{(G,\mathcal{H})}$. Then $G_u$ is 
{\em quadratically hanging} if it is the fundamental group of a surface $\Sigma$ with boundary, 
each incident edge group is a boundary subgroup and every conjugate of a group in $\mathcal{H}$ intersects $G_u$ in a boundary subgroup. 
In this case, we say that a boundary component $C$ of the surface $\Sigma$ is {\em used} if there exists an 
incident edge group or a subgroup of $G_{u}$ conjugate to some $H\in\mathcal{H}$ that is contained in $\pi_1(C)$ as a finite index subgroup.  

The following theorem is essentially due to Sela (non-relative case), but the relative case 
is contained in \cite[Theorem 8.20]{GuirardelLevittJSJI}.

\begin{thm}
Let $G$ be a torsion-free hyperbolic group freely indecomposable with respect to a subgroup $H$. 
Then a cyclic relative $JSJ$-decomposition exists. The flexible vertex groups 
are quadratically hanging and every component of the corresponding surfaces is used. 
\end{thm}

Note that in this case quadratically hanging vertex groups are also called vertex groups of {\em surface type}.

The following lemma will be useful in practice. 

\begin{lemma}\cite[Lemma 5.3]{GuirardelLevittJSJI}\label{JSJ}
Let $G$ be finitely generated. Let $T$ be a universally elliptic $G$-tree. Then $G$ has a $JSJ$-tree if and only if 
every vertex stabilizer $G_{u}$ of $T$ has a relative $JSJ$-tree with respect to $P_{u}$ 
(the set of incident edge stabilizers). Moreover, the $JSJ$-tree is obtained by 
refining $T$ using these trees.
\end{lemma}

\begin{rmk}\cite[Remark 5.4]{GuirardelLevittJSJI}\label{RelJSJ}
Lemma \ref{JSJ} is true in the relative case (where $\mathcal{H}$ is the family of ``fixed'' subgroups of $G$), 
provided that one adds to $P_{u}$ all subgroups of $G_u$ that are conjugate to some $H\in\mathcal{H}$.
\end{rmk}

We give some easy examples of relative cyclic $JSJ$-decompositions. 

\begin{ex} \label{Exa}
\begin{itemize}
 \item[(i)] The relative $JSJ$-decomposition of $\F_2=\langle e_1,e_2\rangle$ with respect to $\langle [e_1,e_2]\rangle$ is the 
 following: $\F_2*_{\langle [e_1,e_2]\rangle}\langle [e_1,e_2]\rangle$.
 \item[(ii)] The relative $JSJ$-decomposition of $\F_4=\langle e_1,e_2,e_3,e_4\rangle$ with respect to $\langle [e_1,e_2][e_3,e_4]\rangle$ is the 
 following: $\F_4*_{\langle [e_1,e_2][e_3,e_4]\rangle}\langle [e_1,e_2][e_3,e_4]\rangle$.
 \item[(iii)] More generally, the relative $JSJ$-decomposition of $\F_{2n}=\langle e_1,\ldots, e_{2n}\rangle$ with respect to 
 $\langle [e_1,e_2][e_3,e_4]\ldots[e_{2n-1},e_{2n}]\rangle$ is the following: $\F_{2n}*_{\langle [e_1,e_2][e_3,e_4]
 \ldots[e_{2n-1},e_{2n}]\rangle}\langle [e_1,e_2]$ $[e_3,e_4]\ldots[e_{2n-1},e_{2n}]\rangle$.
\end{itemize}
\end{ex}
The following pictures give the intuition behind the first two above mentioned examples.

\begin{figure}[ht!]
\hfill
\begin{minipage}[t]{.4\textwidth}
\centering
\includegraphics[width=.3\textwidth]{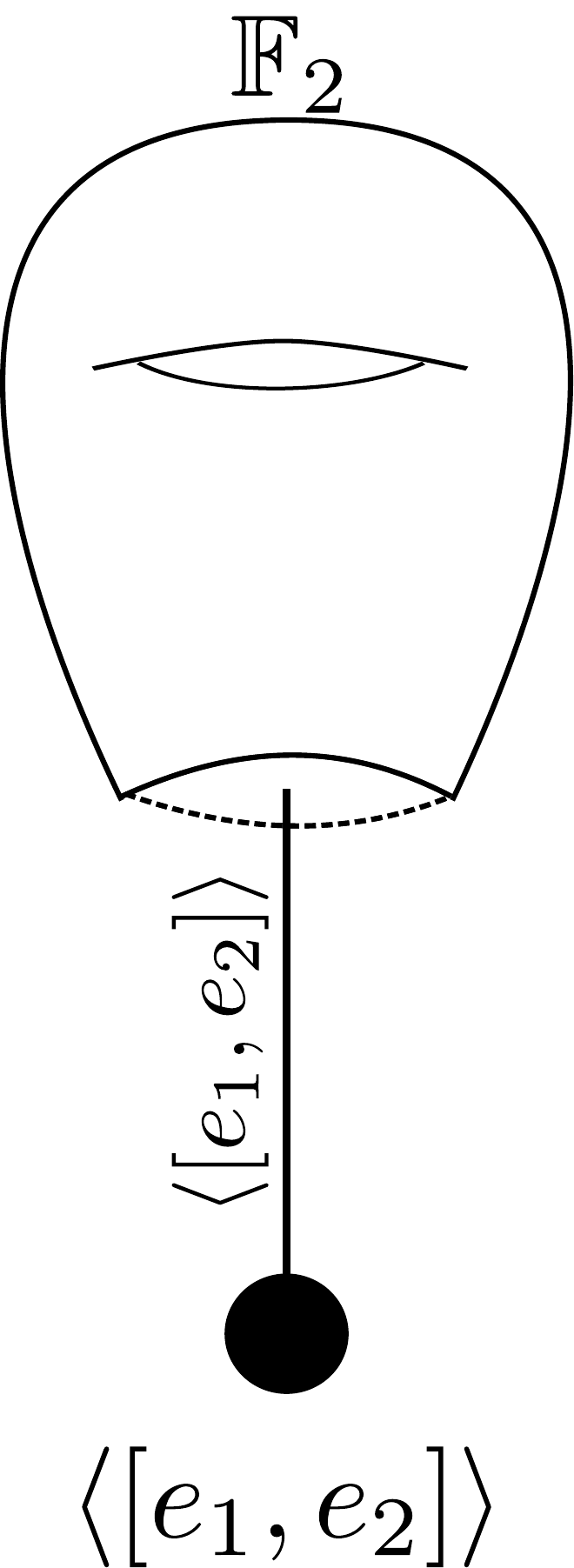}
\caption{The relative $JSJ$ of $\F_2$ with respect to $\langle [e_1,e_2]\rangle$}
\end{minipage}
\hfill
\begin{minipage}[t]{.4\textwidth}
\centering
\includegraphics[width=.35\textwidth]{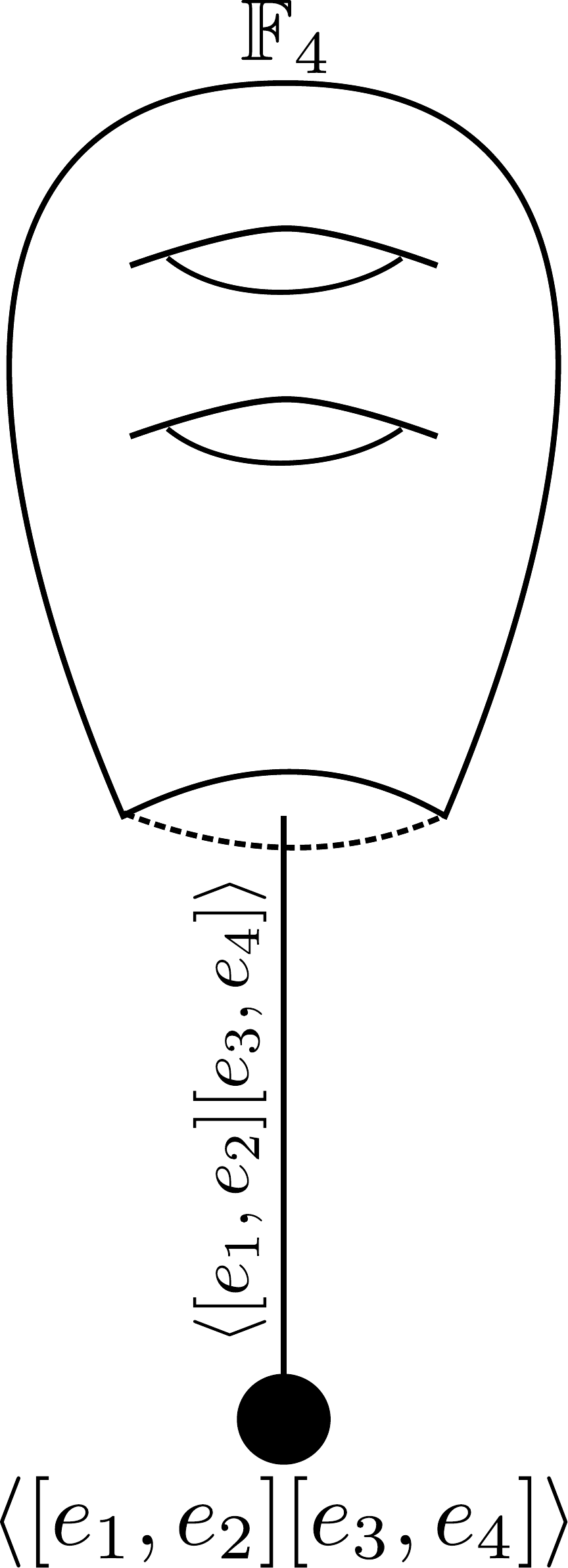}
\caption{The relative $JSJ$ of $\F_4$ with respect to $\langle [e_1,e_2][e_3,e_4]\rangle$}
\end{minipage}
\end{figure}

\section{Imaginaries, Algebraic closure, Forking}\label{3}
The main result that allowed the proof in \cite{AmpleOK} is the geometric elimination of imaginaries due to Sela \cite{SelaImaginaries}. 
Although in this paper we do not use it directly, we state it for completeness. We start by defining some ``tame'' families of imaginaries. 

\begin{defi}
Let $G$ be a torsion-free hyperbolic group. The following equivalence relations in $G$ are called basic.
\begin{itemize}
 \item[(i)] $E_1(a,b)$ if and only if there is $g\in G$ such that $a^g=b$. (conjugation)
 \item[$(ii)_m$] $E_{2_m}((a_1,b_1),(a_2,b_2))$ if and only if $b_1,b_2\neq 1$ and $C_{G}(b_1)=C_{G}(b_2)=\langle b \rangle$ and
$a_1^{-1}a_2\in\langle b^m \rangle$. ($m$-left-coset)
 \item[$(iii)_m$] $E_{3_m}((a_1,b_1),(a_2,b_2))$ if and only if $b_1,b_2\neq 1$ and $C_{G}(b_1)=C_{G}(b_2)=\langle b \rangle$ and
$a_1a_2^{-1}\in\langle b^m \rangle$. ($m$-right-coset)
 \item[$(iv)_{m,n}$] $E_{4_{m,n}}((a_1,b_1,c_1),(a_2,b_2,c_2))$ if and only if
$a_1,a_2,c_1,c_2\neq 1$ and $C_{G}(a_1)=C_{G}(a_2)=\langle a \rangle$ and $C_{G}(c_1)=C_{G}(c_2)=\langle c \rangle$
  and there is $\gamma\in \langle a^n \rangle$ and $\epsilon\in \langle c^n \rangle$ such that $\gamma b_1 \epsilon=b_2$. ($m,n$-double-coset)
\end{itemize}

\end{defi}

Sela proved geometric elimination of imaginaries up to the basic sorts (see \cite{SelaImaginaries}).

\begin{thm}\label{Elim}
Let $G$ be a torsion-free hyperbolic group. Let $E(\bar{x},\bar{y})$ be a $\emptyset$-definable equivalence relation in $G$, with $\abs{\bar{x}}=m$.
Then there exist $k,l<\omega$ and a $\emptyset$-definable relation:
$$R_E \subseteq G^m \times G^k \times S_1(G) \times \ldots \times S_l(G)$$
such that:
\begin{itemize}
 \item[(i)] each $S_i(G)$ is one of the basic sorts;
 \item[(ii)] for each $\bar{a}\in G^m$ , $\abs{R_E(\bar{a},\bar{z})}$ is uniformly bounded (i.e. the bound does not depend on $\bar{a}$);
 \item[(iii)] $R_E(\bar{a},\bar{z})\leftrightarrow R_E(\bar{b},\bar{z})$ if and only if $E(\bar{a},\bar{b})$.
\end{itemize}
\end{thm}

The following theorem of Ould Houcine-Vallino gives an understanding of the algebraic closure in free groups with respect to $JSJ$-decompositions
(see \cite{AlgOV}). 
We recall that given an abelian splitting of $G$, i.e. a splitting in which all edge groups are abelian, 
then the {\em elliptic abelian neighborhood} of a vertex group in this splitting 
is the subgroup generated by the elliptic elements that commute with nontrivial
elements of the vertex group.

\begin{thm}\label{AlgClos}
Let $A$ be a non abelian subgroup of $\F_n$, such that $\F_n$ is freely indecomposable with respect to $A$. Then $acl(A)$ 
coincides with the elliptic abelian neighborhood of the vertex group containing $A$ in the cyclic 
$JSJ$-decomposition of $\F_n$ relative to $A$.
\end{thm}

For the purpose of this paper the following theorems of Pillay concerning forking independence and genericity are enough.
 
\begin{thm}[Corollary 2.7(ii)\cite{PillayForking}]\label{ForkPil}
Any basis of $\F_n$ is an independent set.
\end{thm}

\begin{thm}[Theorem 2.1(i)\cite{PillayGenericity}]\label{GenPil}
Suppose $a$ is a generic element in $\F_n$. Then $a$ is primitive.
\end{thm}

For completeness we give the description of forking independence over free factors given by Perin-Sklinos \cite{PerSklFork}.

\begin{thm}\label{Fork}
Let $\bar{a},\bar{b}\in \F_n$ and $G$ be a free factor (possibly trivial) of $\F_n$. Then 
$\bar{a}$ does not fork with $\bar{b}$ over $G$ if and only if $\F_n$ admits a free decomposition $\F_n=\F*G*\F'$ and 
$\bar{a}\in \F*G$ and $\bar{b}\in G*\F'$.
\end{thm}

We will also use the following theorems from \cite{AmpleOK}. We note that $acl^c(A)$ denotes the set of conjugacy classes in 
$acl^{eq}(A)$.

\begin{thm}\label{AlgConj}
Let $\bar{a},\bar{b},\bar{g}$ be finite tuples from $\F_n$. Then we have:
$$ acl^{eq}(\bar{a})\cap acl^{eq}(\bar{b})=acl^{eq}(\bar{g})$$ 
if and only if 
$$ acl(\bar{a})\cap acl(\bar{b})= acl(\bar{g})$$ 
and 
$$ acl^c(\bar{a})\cap acl^c(\bar{b})=acl^c(\bar{g})$$
\end{thm}

\begin{thm}\label{conj}
Let $A$ be a non abelian subgroup of $\F_n$, such that $\F_n$ is freely indecomposable with respect to $A$. Let $b\in\F_n$. Then the conjugacy 
class of $b$ belongs to $acl^c(A)$ 
if and only if in any cyclic $JSJ$-decomposition of $\F_n$ relative to $A$, either $b$ is conjugate to some element of 
the elliptic abelian neighborhood of a rigid vertex group or it is a conjugate to an element of a boundary subgroup of a 
surface type vertex group.
\end{thm}

\section{Witnessing Ampleness}
The following sequence in $\F_{\omega}$ witnesses $n$-ampleness, for any $n<\omega$. We give the sequence recursively:

$$a_0=e_1$$ 
$$a_{i+1}=a_i[e_{2i+2},e_{2i+3}], \ \textrm{for}\ 0\leq i<\omega$$

We fix a natural number $n\geq 1$, and we show that $a_0,\ldots, a_n$ witnesses $n$-ampleness 
by verifying the requirements of Definition \ref{Ample}. 

\begin{lemma}\label{D1}
$a_0=e_1$ forks with $a_n=e_1[e_2,e_3]\ldots[e_{2n},e_{2n+1}]$.
\end{lemma}

\begin{proof}
Immediate, since $[e_2,e_3]\ldots[e_{2n},e_{2n+1}]$ is not a primitive element, thus by Theorem \ref{GenPil} 
is not generic.
\end{proof}

\begin{lemma}\label{D2}
Let $1\leq i<n$. Then $a_0,\ldots,a_{i-1}$ does not fork with $a_{i+1}$ over $a_i$. 
\end{lemma}

\begin{proof}
 We first note that for each $i$, $\langle a_i\rangle$ is a free factor of $\F_{2i+3}$. Thus, by Theorem \ref{ForkPil}, we only need to find a 
 free factorization $\F_{2i+3}=\F*\langle a_i\rangle* \F'$, such that $a_0,\ldots,a_{i-1}$ is in 
 $\F*\langle a_i\rangle$ and $a_{i+1}$ is in $\langle a_i\rangle* \F'$. It is easy to see that the following free factorization is such
 $\F_{2i+3}=\langle e_2,\ldots, e_{2i},e_{2i+1}\rangle *\langle a_i\rangle * \langle e_{2i+2},e_{2i+3}\rangle$.
\end{proof}

\begin{lemma}\label{D3}
$acl^{eq}(e_1)\cap acl^{eq}(e_1[e_2,e_3])=acl^{eq}(\emptyset)$. 
 \end{lemma}

\begin{proof}
It is not hard to see that $acl(e_1)=\langle e_1 \rangle$ and $acl^c(e_1)$ is the set of conjugacy classes of elements in 
$\langle e_1 \rangle$. The same is true for $acl(e_1[e_2,e_3])$ and $acl^c(e_1[e_2,e_3])$.  
Thus, $acl(e_1)\cap acl(e_1[e_2,e_3])=acl(\emptyset)$ and $acl^c(e_1)\cap acl^c(e_1[e_2,e_3])=acl^c(\emptyset)$. 
By Theorem \ref{AlgConj} we are done. 

\end{proof}

To verify Definition \ref{Ample} (4) we first compute the $JSJ$-decomposition of $\F_{2i+1}$ relative to $\langle a_0,\ldots, a_{i-1}, a_i\rangle$ 
and the $JSJ$-decomposition of $\F_{2i+3}$ relative to $\langle a_0,\ldots, a_{i-1}, a_{i+1}\rangle$. We give the following pictures:

\begin{figure}[h!]
\centering
\includegraphics[width=.5\textwidth]{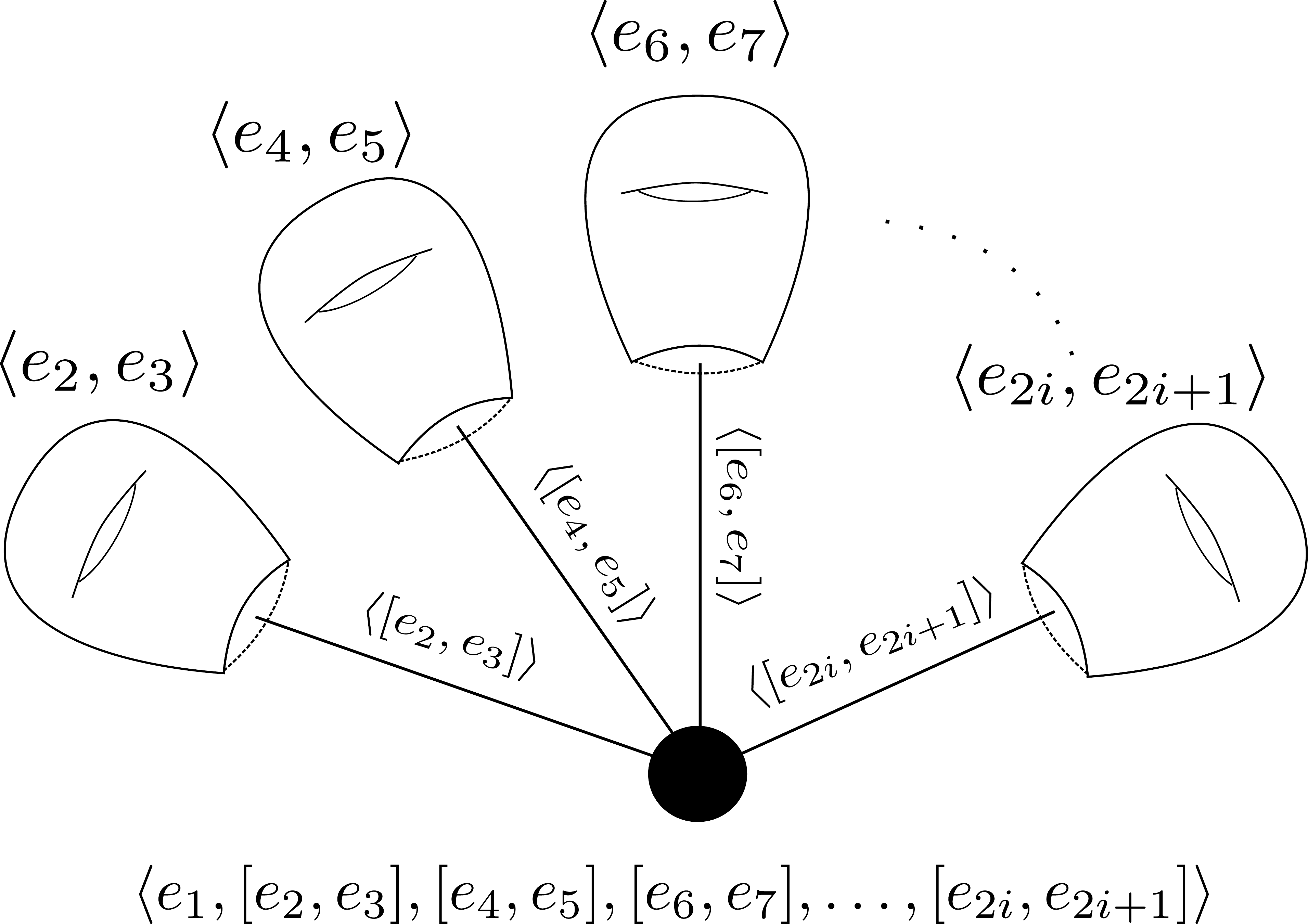}
\caption{The relative $JSJ$-decomposition of $\F_{2i+1}$ with respect to $\langle a_0,\ldots, a_{i-1}, a_i\rangle$}
\end{figure}

\begin{figure}[h!]
\centering
\includegraphics[width=.75\textwidth]{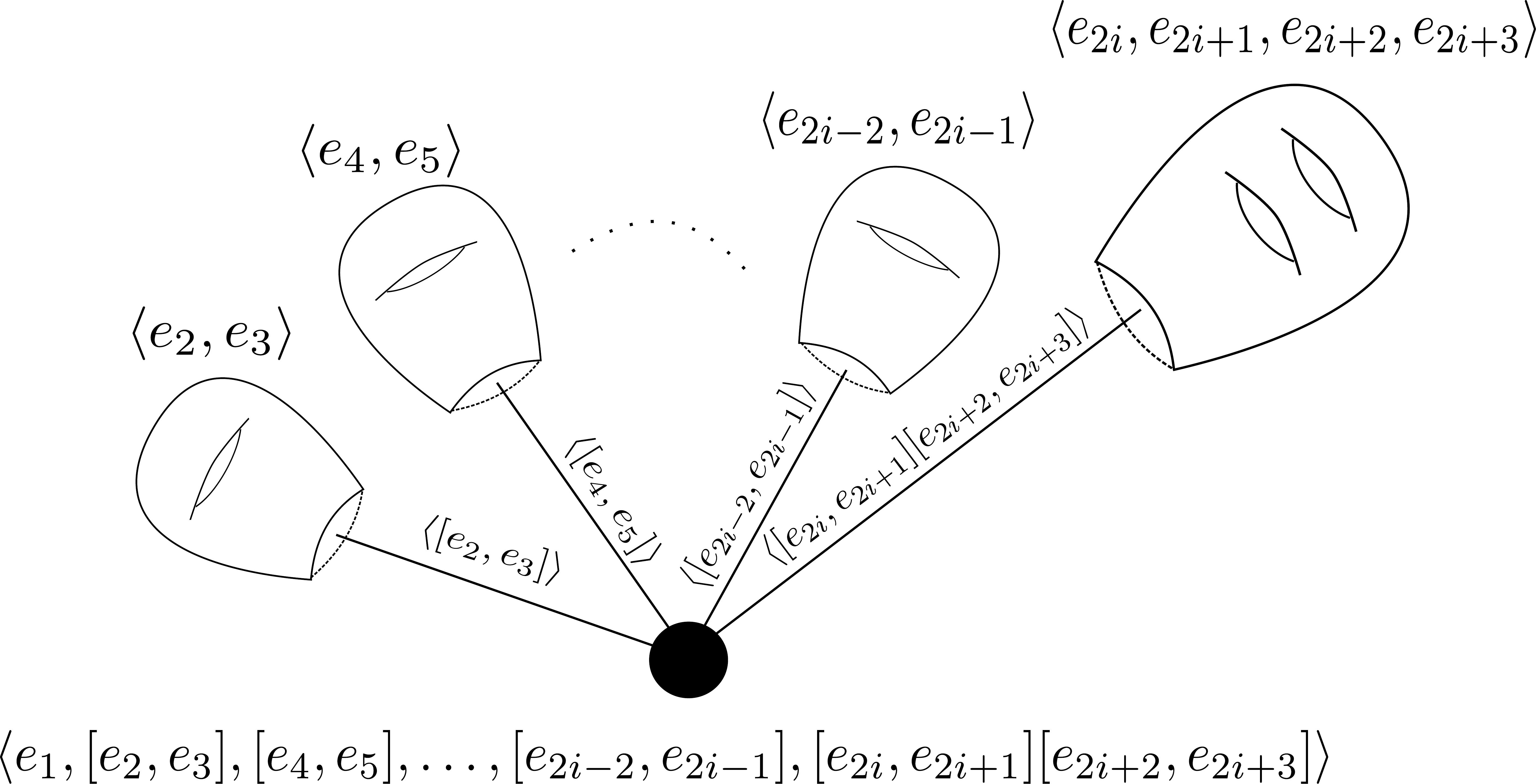}
\caption{The relative $JSJ$-decomposition of $\F_{2i+3}$ with respect to $\langle a_0,\ldots,a_{i-1}, a_{i+1}\rangle$}
\end{figure}

\begin{lemma}\label{gog1}
The $JSJ$-decomposition of $\F_{2i+1}$ relative to $\langle a_0,\ldots, a_{i-1}, a_i\rangle$ is the graph of groups 
given by figure $3$.
\end{lemma}
\begin{proof}
It is immediate that the given splitting is universally elliptic. That is a $JSJ$-decomposition 
follows from Remark \ref{RelJSJ} and Example \ref{Exa} (i).
\end{proof}

\begin{lemma}\label{gog2}
The $JSJ$-decomposition of $\F_{2i+3}$ relative to $\langle a_0,\ldots, a_{i-1}, a_{i+1}\rangle$ is the graph of groups 
given by figure $4$.
\end{lemma}
\begin{proof}
It is immediate that the given splitting is universally elliptic. That is a $JSJ$-decomposition 
follows from Remark \ref{RelJSJ} and Example \ref{Exa} (i) and (ii).
\end{proof}

We are now ready to finish our proof.

\begin{lemma} \label{D4}
Let $1\leq i<n$. Then $acl^{eq}(a_0,\ldots,a_{i-1},a_i)\cap acl^{eq}(a_0,\ldots,a_{i-1},$ $a_{i+1})=acl^{eq}(a_0,\ldots,a_{i-1})$.
\end{lemma}

\begin{proof}
We first note that by Theorem \ref{AlgClos} and Lemmata \ref{gog1},\ref{gog2} we have that $acl(a_0,\ldots,a_{i-1},a_i)=\langle a_0,\ldots,a_{i-1},a_i \rangle$ 
and $acl(a_0,\ldots,a_{i-1},a_{i+1})=\langle a_0,\ldots,a_{i-1},a_{i+1} \rangle$. Thus, their intersection is $\langle a_0,\ldots,a_{i-1}\rangle$ which is 
exactly $acl(a_0,\ldots,a_{i-1})$. 

By Theorem \ref{AlgConj} we only need to show that 
$acl^c(a_0,\ldots,a_{i-1},a_i)\cap acl^c(a_0,$ $\ldots,a_{i-1},a_{i+1})=acl^c(a_0,\ldots,a_{i-1})$. But, by Theorem \ref{conj} and 
Lemmata \ref{gog1},\ref{gog2} we have that $acl^c(a_0,\ldots,a_{i-1},a_i)$ is exactly the set of conjugacy classes of elements in 
$\langle a_0,\ldots,a_{i-1},a_i \rangle$ and $acl^c(a_0,\ldots,a_{i-1},a_{i+1})$ is exactly the set of conjugacy classes of elements 
in $\langle a_0,\ldots,a_{i-1},a_{i+1}\rangle$. Thus, their intersection is the set of conjugacy classes of elements in $\langle a_0,\ldots,a_{i-1}\rangle$, 
which is exactly $acl^c(a_0,\ldots,a_{i-1})$.
\end{proof}

Putting everything together, our Lemmata \ref{D1},\ref{D2},\ref{D3},\ref{D4}, show that our sequence $(a_i)_{i<\omega}$ witnesses 
$n$-ampleness in the theory of non abelian free groups for each $n<\omega$. And thus giving an alternative sequence to 
the one used in \cite{AmpleOK}.

\begin{thm}
$T_{fg}$ is $n$-ample for each $n<\omega$.
\end{thm}

\begin{rmk}
Actually we have produced a family of sequences witnessing the $n$-ampleness of $T_{fg}$ for each $n<\omega$ as 
instead of once punctured tori we could have used any other once punctured surface (with a few exceptions of some obvious small surfaces).
\end{rmk}

\paragraph{Acknowledgements.} We wish to thank Chlo\'e Perin for a thorough reading of a first version of this paper and for spotting a 
misquotation.

\bibliography{biblio}
\ \\ \\ 
Hebrew University of Jerusalem,\\
Einstein Institute of Mathematics,\\
91904, Israel\\
rsklinos@math.huji.ac.il

\end{document}